\documentclass{article}  
\pdfoutput=1
\usepackage[style=numeric, backend=bibtex]{biblatex}
\usepackage{amsmath}
\usepackage{amssymb}
\usepackage{amsthm}
\usepackage{stmaryrd}
\usepackage[shortlabels]{enumitem}
\usepackage{mathtools}
\usepackage[hidelinks]{hyperref}
\usepackage{xcolor}
\usepackage[margin=1in]{geometry}
\usepackage{float}
\usepackage{tikz-cd}

\makeatletter

\newenvironment{claimproof}{
	\textit{Proof of Claim.}
}{
	\hfill$\diamond$
}

\newcommand{\ifftext}{\quad \text{if and only if} \quad}

\DeclareMathOperator{\ord}{ord}

\DeclareMathOperator{\Gal}{Gal}

\DeclareMathOperator{\charK}{char}
\DeclareMathOperator{\rank}{rank}

\newcommand{\IF}{\mathbb{F}}
\newcommand{\IZ}{\mathbb{Z}}
\newcommand{\IN}{\mathbb{N}}
\newcommand{\IQ}{\mathbb{Q}}
\newcommand{\Oo}{\mathcal{O}}
\newcommand{\Mm}{\mathcal{M}}
\newcommand{\Aa}{\mathcal{A}}

\theoremstyle{definition}
\newtheorem{definition}{Definition}[section]
\newtheorem{theorem}[definition]{Theorem}
\newtheorem{lemma}[definition]{Lemma}
\newtheorem{corollary}[definition]{Corollary}
\newtheorem{fact}[definition]{Fact}
\newtheorem{remark}[definition]{Remark}

\title{An Elementary Proof of the Local Kronecker-Weber Theorem}
\author{Jochen Koenigsmann, Benedikt Stock}

\begin{document}
\maketitle
\begin{abstract}
	We will present a novel elementary, self-contained, and explicit proof of the local Kronecker-Weber theorem. Apart from discrete valuation theory, it does not make use of any tools beyond those introduced in a second undergraduate course on algebra. In particular, we will not make use of results from local class field theory or Galois cohomology.
\end{abstract}


The Kronecker--Weber theorem states that every finite abelian extension of the rational numbers is contained in a cyclotomic extension.  A number of proofs of this result are known. Culler~\cite{Culler07} presents two proofs; an elementary approach based on David Hilbert's original proof and one that makes use of class field theory. Another classical approach proceeds via the local Kronecker–Weber theorem---the subject of this paper---in which $\IQ$ is replaced by $\IQ_p$. Once the local version is established, it is not hard to deduce the global version via completions, Minkowski's bound, and basic properties of discriminants (cf.~\cite[Chapter 14]{Guillot18} for an excellent presentation of a complete proof). 

As in the global case, several proofs of the local Kronecker--Weber theorem are known. Perhaps the most classical one goes via the fundamental theorem of local class field theory, a result that itself is usually either proved via the cohomological approach (cf.~\cite{Guillot18}) or Lubin--Tate theory (cf.~\cite{Childress08, Hazewinkel75}). Our proof is distinct from those as it makes no mention of local class field theory at all. Papers by Rosen~\cite{Rosen81} and Kozuka~\cite{Kozuka90} have a similar goal in mind. However, the new approach presented here differs from these in that it is self-contained and does not even mention Lubin--Tate constructions or Galois cohomology. Instead, our considerations are merely based on discrete valuation theory combined with basic results from Galois theory and Kummer theory. Moreover, it is explicit in that it provides a formula for a cyclotomic extension containing the given abelian extension, described in terms of the degree of the abelian extension. A proof with a similar structure to ours can be found in Washington's classical book~\cite[Chapter 14]{Washington97}; we will discuss the key similarities and differences in more detail at the end of this paper in Remark~\ref{rem:Washington}.

In  Section~\ref{sec:facts} we will remind the reader of some facts about local fields before presenting the proof of the local Kronecker--Weber theorem in Section \ref{sec:proof}.

\section{Some facts about local number fields}
\label{sec:facts}
Let us first fix some notation. We call finite extensions of $\IQ_p$ \emph{local number fields} or \emph{$p$-adic fields}. For a $p$-adic field $F$, we denote by $v_F$ (or simply $v$ when the field is clear from context) the unique extension of the $p$-adic valuation $v_p$ to $F$. We write $\Oo_F$ for the ring of integers of $F$, $\Mm_F$ for its unique maximal ideal, and $\overline{F}$ for its residue field $\Oo_F / \Mm_F$. 

For an extension $L/F$ of local number fields of degree $n$, we denote by $e = e(L/F)$ and $f = f(L/F)$ the \emph{ramification index} and \emph{inertia degree} of the extension $L/F$, respectively. We recall the fundamental equality $n = ef$ (cf.~\cite[Theorem~2.38]{Guillot18}, \cite[II.~6.8]{Neukirch99}). More precisely, there exists an intermediate field $L_0$, called the \emph{inertia subfield} of $L$, with the property that for any intermediate field $F \subseteq K \subseteq L$, the extension $K/F$ is unramified if and only if $K \subseteq L_0$ (cf.~\cite[Corollary~2.43]{Guillot18}).

\subsection{Cyclotomic extensions of local fields}
We introduce some additional notation. For a field $F$, let $\mu_F$ denote the set of all roots of unity contained in $F$. Whenever $F$ contains all $n$-th roots of unity for some $n \in \mathbb{N}_{>0}$, we write this set as~$\mu_n$. A fixed primitive $n$-th root of unity is denoted by~$\zeta_n$, and we write $\zeta_{q^{\infty}}$ for the set of all $q$-power roots of unity. Finally, when several roots of unity are involved, we tacitly assume that they form a compatible system (for example, $\zeta_{q^2}^q = \zeta_q$).

Recall that the cyclotomic extensions of~$\IQ_p$ admit explicit descriptions (see~\cite[Propositions~2.46 and~2.48]{Guillot18}), which are summarised in the following lemma.

\begin{lemma}
	\label{lem:cycExt}
	Let $F = \IQ_p(\mu_m)$ be a cyclotomic extension of $\IQ_p$.
	\begin{enumerate}[(i)]
		\item If $(m,p) = 1$, then $F/\IQ_p$ is unramified of degree $k = \ord_m(p)$, the order of $p +m\IZ$ in $(\IZ/m\IZ)^\times$. Moreover,
		\[
		\Gal(F/\IQ_p) \cong \Gal(\IF_{p^k}/\IF_p) \cong C_k.
		\]
		\item If $m = p^l$, then $F/\IQ_p$ is totally ramified of degree $(p - 1)p^{l - 1}$ and $1 - \zeta_{p^l}$ is a uniformiser. Moreover,
		\[
		\Gal(F/\IQ_p) \cong \bigl(\IZ/p^l\IZ\bigr)^\times.
		\]
	\end{enumerate}
\end{lemma}

\begin{proof}[Proof (Sketch)]
	(i). We first observe that $\zeta_m \in \Oo_F$ since $\zeta_m^m=1$ implies $v_F(\zeta_m)=0$. In particular, the minimal polynomial $f(X)$ of $\zeta_m$ over $\IQ_p$ lies in $\IZ_p[X]$, so $\overline{f}(X)\in\IF_p[X]$, where $\overline{f}(X)$ divides $X^m-1$. Since $(m,p)=1$, this means $\overline{f}(X)$ is separable over $\IF_p$, and hence irreducible by Hensel's Lemma. In particular, $n=f$ and $\overline{F} = \IF_p(\overline{\zeta_m})$. By the fundamental equality $n=ef$, we obtain $e=1$, so $F/\IQ_p$ is unramified.
	
	By the Galois theory of finite fields, $\Gal(\overline{F}/\IF_p)$ is cyclic of order $k$, generated by the Frobenius automorphism $x\mapsto x^p$, which is determined entirely by its action on $\overline{\zeta_m}$. The order of the Frobenius automorphism is therefore the smallest positive number $k$ for which $\overline{\zeta_m}^{p^k}=\overline{\zeta_m}$, i.e., $p^k\equiv1\pmod m$, so $k = \ord_m(p)$. Finally, using again that any element in $\Gal(\overline{F}/\IF_p)$ is determined by the image of $\overline{\zeta_m}$, every such element lifts to an element of $\Gal(F/\IQ_p)$, which means the natural restriction map 
	\[ 
	\Gal(F/\IQ_p)\to\Gal(\overline{F}/\IF_p) 
	\]
	is surjective. Since $f=n$, this restriction is in fact an isomorphism of Galois groups, from which the remaining claims follow. 
	
	(ii). Define 
	\[
	\Phi(X) \coloneq 1 + X^{p^{l-1}} + \cdots + X^{(p-1)p^{l-1}}.
	\]
	It is easy to see that $\Phi(\zeta_{p^l}) =0$. Moreover, applying the Eisenstein criterion to $\Phi(X+1)$, we see that $\Phi(X)$ is irreducible, which means $\Phi$ is the minimal polynomial of $\zeta_{p^l}$ over $\IQ_p$. The other roots of $\Phi$ are exactly the $\phi(p^l)= (p-1)p^{l-1}$ primitive $p^l$-th roots of unity of the form $\zeta_{p^l}^i$ for $(i, p) = 1$. Hence, the elements of $\Gal(F/\IQ_p)$ are precisely determined by $\zeta_{p^l} \mapsto \zeta_{p^l}^i$, from which we recover the usual isomorphism $\Gal(F/\IQ_p) \cong (\IZ/p^l\IZ)^\times$.
	
	Next, we observe that for every $\sigma \in \Gal(F/\IQ_p)$, $v(\sigma(x)) = v(x)$ for all $x \in F$ by the uniqueness of $v$. Hence, using
	\[
	p = \Phi(1)  = \prod_{\sigma \in \Gal(F/\IQ_p)}\sigma(1-\zeta_{p^l}),
	\]
	we see that $v(1-\zeta_{p^l}) = \frac{v(p)}{[F:\IQ_p]}$. Thus, $1-\zeta_{p^l}$ is the required uniformiser, and it also witnesses that $F/\IQ_p$ is totally ramified.
\end{proof}
From the fact that $\IF_p^\times = \mu_{p - 1}$, we obtain the following description of the roots of unity in $\IQ_p$.
\begin{corollary} \label{cor:rootsOfUnity}
	In $\IQ_p$, the roots of unity are
	\[
	\mu_{\IQ_p} = \begin{cases}
		\mu_{p - 1} & \text{if $p \ne 2$}; \\
		\mu_2 & \text{if $p = 2$.}
	\end{cases}
	\]
\end{corollary}

\subsection{Residues and generators in \texorpdfstring{$\IQ_p$}{Qp}}
The goal of this section is to study the multiplicative structure of $p$-adic fields by explicit means.
Our starting point is the filtration
\[
F^\times \supseteq U_F^{(0)} \supseteq U_F^{(1)} \supseteq \cdots \supseteq U_F^{(k)} \supseteq \cdots
\]
of the multiplicative group $F^\times$ of a $p$-adic field $F$ by \emph{higher unit groups}.

\begin{definition}
	Let $F$ be a $p$-adic field with uniformiser $\pi$. The \emph{$k$-th higher unit group} ($k \in \IN$) is defined as
	\begin{align*}
		U^{(0)}_F & \coloneqq \Oo_F^\times; \\
		U^{(k)}_F & \coloneqq \bigl\{x \in \Oo_F^\times: x \equiv 1 \; (\text{mod }{\pi^k})\bigr\} \quad \text{for $k \ge 1$.}
	\end{align*}
\end{definition}

For the remainder of this section,  we display some calculations involving higher unit groups, some of which are standard and others which may be less well-known.

\begin{fact}
	\label{fact:multgrp_direct_prod}
	Let $F$ be a $p$-adic field with uniformiser $\pi$. Denote its prime-to-$p$ roots of unity by $\mu'_F$. The multiplicative group $F^\times$ can be decomposed as the direct product
	\[
	F^\times = \pi^{\IZ} \times \mu'_F \times U_F^{(1)}.
	\]
\end{fact}
\begin{proof}
	Any $x \in F^\times$ can be uniquely written as $x = \pi^{v(x)}(\pi^{-v(x)}x)$, where $\pi^{-v(x)}x \in \Oo_F^\times$. Therefore, $F^\times =  \pi^{\IZ} \times \Oo_F^\times$. Now,
	\begin{equation}
		\label{eq:exact_seq_U1}
		1 \rightarrow U_F^{(1)} \rightarrow \Oo_F^\times \rightarrow  \overline{F}^\times  \rightarrow 1
	\end{equation}
	is split via the section $\overline{\omega} \mapsto \omega$ for each $\omega \in \mu'_F$, so $\Oo_F^\times = \mu'_F \times U_F^{(1)}$ as required.
\end{proof}
\begin{fact}
	\label{fact:TowerLawU}
	Let $F$ be a $p$-adic field. For any $k \ge 1$, there exist isomorphisms
	\begin{align*}
		U_F^{(0)}/U_F^{(1)} & \cong (\overline{F}^\times, \cdot) \\
		U_F^{(k)}/U_F^{(k+1)} & \cong (\overline{F},+).
	\end{align*}
\end{fact}
\begin{proof}
	The first isomorphism is an immediate consequence of the short exact sequence \eqref{eq:exact_seq_U1}. For $k \ge 1$, the map
	\begin{align*}
		U_F^{(k)}/U_F^{(k+1)} & \longrightarrow \overline{F}, \\
		(1 + \pi^k a)U_F^{(k + 1)} & \longmapsto a+\Mm_F
	\end{align*}
	is a well-defined group isomorphism; see \cite[Chapter~II, (3.10)]{Neukirch99}, \cite[Lem\-ma~2.56]{Guillot18}, and Remark~\ref{rem:mod_rule} below.
\end{proof}

\begin{remark} \label{rem:mod_rule}
	The fact that the above maps are indeed isomorphisms relies on the following rule linking additive and multiplicative congruences: for any $a, a' \in U_F^{(0)}$,
	\[
	a \equiv a' \pmod{\pi^k} \ifftext a \equiv a' \quad\bigl(\mathrm{mod}\ U_F^{(k)}\bigr).
	\]
	The equivalence is easily verified:
	\[
	\exists b \in \Oo_F : a' = a + b\pi^k \Longleftrightarrow \exists b \in \Oo_F : \frac{a'}a = 1 + a^{-1}b\pi^k \Longleftrightarrow \frac{a'}a \in U_F^{(k)}.
	\]
\end{remark}

\begin{corollary} 
	\label{cor:representation}
	Let $F$ be a $p$-adic field and $\Aa \subseteq \Oo_F$ be a set of representatives for $\overline{F} = \Oo_F/\Mm_F$. Then any $x \in U_F^{(1)}/U_F^{(k+1)}$ has a unique representation of the form $x = aU_F^{(k+1)}$, where
	\[
	a = 1 + a_1 \pi + \cdots + a_k \pi^k,
	\]
	and all coefficients $a_i$ lie in $\Aa$.
\end{corollary}
\begin{proof}
	Assume that $a = 1 + \sum_{i = 1}^k a_i\pi^i$ and $a' = 1 + \sum_{i = 1}^k a_i'\pi^i$ represent the same element modulo $U_F^{(k + 1)}$. In other words, there exists $1 + \pi^{k + 1}b \in U_F^{(k + 1)}$ with $a = a'(1 + \pi^{k + 1}b)$. Assume $a_i \ne a_i'$ for some $i$, in which case we choose the minimal such $i$. Then
	\[
	v(a - a') = v((a_i - a_i') \pi^i) = v(\pi^i),
	\]
	which contradicts $a - a' = a' \pi^{k + 1}b$. Therefore, $a_i = a_i'$ for all $i = 1, \ldots, k$.
	
	The fact that any element in $U_F^{(1)}/U_F^{(k+1)}$ can be written as claimed follows from a counting argument. Observe that
	\[
	\bigl|U_F^{(1)}/U_F^{(k+1)}\bigr| =\bigl|U_F^{(1)}/U_F^{(2)}\bigr|\cdots \bigl|U_F^{(k)}/U_F^{(k+1)}\bigr| = |\overline{F}|^k
	\]
	by Fact~\ref{fact:TowerLawU}, which is exactly the number of distinct representatives we consider.
\end{proof}

The following residue calculation is crucial for many of the results that follow.

\begin{lemma} \label{lem:p/(1-zeta)^p}
	Let $F$ be a $p$-adic field containing $\zeta_p$ and let $\pi = \zeta_p - 1$. Then
	\[
	\frac{p}{\pi^{p-1}} \equiv - 1 \pmod{\pi}.
	\]
	In particular, $p/\pi^{p - 1}$ reduces to $-1$ in the residue field.
\end{lemma}
\begin{proof}
	Consider
	\[
	f(X) = X^{p - 1} + \cdots + X + 1 = \prod_{i=1}^{p-1}(X-\zeta_p^i).
	\]
	Observe that
	\[
	\frac{p}{\pi^{p-1}} =\frac{f(1)}{(-1)^{p-1}(1-\zeta_p)^{p-1}} = (-1)^{p-1}\prod_{i=1}^{p-1}\biggl(\frac{1-\zeta^i_p}{1-\zeta_p}\biggr) = (-1)^{p-1}\prod_{i=1}^{p-1}\sum_{j=0}^{i-1}\zeta_p^j.
	\]
	For each factor, we compute
	\[
	\sum_{j=0}^{i-1}\zeta_p^j = \sum_{j=0}^{i-1}(1+\pi)^j \equiv \sum_{j=0}^{i-1}1 = i \pmod{\pi}.
	\]
	The claim then follows using that $p \in (\pi)$, $(-1)^{p-1} \equiv 1 \pmod{p}$, and Wilson's theorem $(p - 1)! \equiv -1 \pmod{p}$.
\end{proof}

The three following results and their proofs are slight adaptations of an idea first observed by Koenigsmann \cite[Lemma 3.2]{Koenigsmann03}.
\begin{lemma}
	\label{lem:U_Kpowers_general}
	Let $F/\IQ_p(\zeta_p)$ be a finite extension with ramification index $e \coloneqq e(F/\IQ_p(\zeta_p))$. Then
	\[
	U_F^{(ep+1)} \subseteq \bigl(U_F^{(e)}\bigr)^p.
	\]
\end{lemma}
\begin{proof}
	Recall that $\IQ_p(\zeta_p)/\IQ_p$ is totally ramified of degree $p - 1$ by Lemma \ref{lem:cycExt}. Let $\pi = \zeta_p - 1$ be a uniformiser for $\IQ_p(\zeta_p)$ and $\Pi$ any uniformiser for $F$. If the valuation $v$ on $F$ is normalised such that $v(p) = 1$, then $v(\pi) = \tfrac 1{p - 1}$ and $v(\Pi) = \tfrac 1{e(p - 1)}$. In particular,
	\[
	U_F^{(e)} = 1 + (\pi) \quad\text{and}\quad
	U_F^{(ep + 1)} = 1 + (\pi^p\Pi).
	\]
	Let $x = 1 + \pi^p a$ with $a \in (\Pi)$ be given. Then it suffices to show that
	\[
	(1 + \pi X)^p - (1 + \pi^p a) = 0
	\]
	has a solution in $\Oo_F$. Consider
	\[
	f(X) \coloneqq \frac{1}{\pi^p}\Bigl[(1 + \pi X)^p - (1 + \pi^pa)\Bigr] = 
	X^p - a + \sum_{i = 1}^{p - 1} \binom p i \pi^{i - p} X^i.
	\]
	Noting that
	\[
	p \pi^{1 - p}  \equiv -1 \pmod{\pi}
	\]
	for $i=1$ by Lemma~\ref{lem:p/(1-zeta)^p}, and
	\begin{equation*}
		v\left[\binom p i \pi^{i - p}\right]  = v(p)+ (i-p)v(\pi)= (i-1)v(\pi) > 0  
	\end{equation*}
	for $i = 2, \ldots, p - 1$, we conclude that $f(X)$ reduces to $\overline{f}(X) = X^p - X \in \overline{F}[X]$. Thus, the claim follows by Hensel's Lemma.
\end{proof}

In the above proof, we could also start with an element $x = 1 + \pi^p a$ with $a \in \Oo_F$ instead of $a \in (\Pi)$. In that case, $x\in F^{p}$ is equivalent to the existence of a solution to the equation 
\[
X^p - (1+\pi^pa) = 0.
\]
Applying a change of variables $X \mapsto 1+\pi X$, and using the same arguments as before, this becomes equivalent to $X^p-X-\overline{a}$ having a solution in $\overline{F}$ (in fact, a posteriori, this even shows equivalence with $x \in \bigl(U_F^{(e)}\bigr)^p$). We record this fact, which is essentially Koenigsmann's original statement, in the following lemma.
\begin{lemma}
	\label{lem:powersConverse}
	Let $F/\IQ_p(\zeta_p)$ be a finite extension and $a \in \Oo_F$. Then
	\[
	1 + \pi^pa \in F^p\iff \exists z \in \overline{F} : z^p-z-\overline{a} = 0.
	\]
\end{lemma}

Next, we observe that the inclusion stated in Lemma~\ref{lem:U_Kpowers_general} becomes an equality in the special case $F = \IQ_p(\zeta_p)$.
\begin{lemma}
	\label{lem:U_Kpowers}
	For $F = \IQ_p(\zeta_p)$, we have 
	\[
	U_F^{(p + 1)} = \bigl(U_F^{(1)}\bigr)^p.
	\]
\end{lemma}
\begin{proof}
	The inclusion $U_F^{(p + 1)} \subseteq \bigl(U_F^{(1)}\bigr)^p$ was shown in Lemma~\ref{lem:U_Kpowers_general}. For the reverse, let $(1 + \pi b)^p \in \bigl(U_F^{(1)}\bigr)^p$ for $b \in \Oo_F$. Then the considerations in the proof of Lemma~\ref{lem:U_Kpowers_general} yield
	\[
	(1 + \pi b)^p \equiv 1 - \pi^p b + \pi^p b^p = 1 + \pi^p(b^p - b) \pmod{\pi^{p + 1}}.
	\]
	Since $\overline{F} = \IF_p$, we have $b^p \equiv b \pmod{\pi}$, and therefore
	\[
	\bigl(U_F^{(1)}\bigr)^p \subseteq U_F^{(p + 1)}. \qedhere
	\]
\end{proof}
The following result synthesises many of the considerations above, and it will play a central role in our proof of the local Kronecker--Weber theorem.

\begin{lemma}
	\label{lem:dimOfPowerQuotientWEAK}
	Let $F = \IQ_p(\zeta_p)$ and let $\pi$ be a uniformiser for $F$. Then the elements of $F^{\times}/F^{\times p}$ can be uniquely represented by
	\[
	a = p^{a_0}(1+a_1\pi + \cdots +a_p\pi^{p}),
	\]
	for some integers $0 \leq a_i \leq p-1$; in particular, $\dim_{\IF_p} F^{\times}/F^{\times p} = p+1$. 
	
	Moreover, for any $b\coloneq c \omega u$ with $c \in \pi^{\IZ}$, $\omega \in \mu_F'$ and $u \in U_{F}^{(1)}$,
	\[
	b \equiv a \hspace{-0.15cm}\pmod{F^{\times p}} \iff  c \equiv p^{a_0}\hspace{-0.15cm} \pmod{\pi^{p\IZ}}\,\land\, u\equiv 1+a_1\pi + \cdots +a_p\pi^{p} \hspace{-0.15cm} \pmod{\pi^{p+1}}.
	\]
	\end{lemma}
	\begin{proof}
		From Fact~\ref{fact:multgrp_direct_prod} and Lemma~\ref{lem:U_Kpowers}, we see that
		\begin{equation}
			\label{eq:F times mod F times p}
			F^\times/F^{\times p} \cong \pi^{\IZ}/\pi^{p \IZ}\times U_F^{(1)}/\big(U_F^{(1)}\big)^p = \pi^{\IZ}/\pi^{p \IZ}\times U_F^{(1)}/U_F^{(p+1)}.
		\end{equation}
		Now, $\{p\}$ induces a basis of $\pi^{\IZ}/\pi^{p\IZ}$, since
		\[
		p \equiv -\pi^{p-1} \pmod{\pi^p}
		\]
		by Lemma~\ref{lem:p/(1-zeta)^p}, and $\{-\pi^{p-1}\cdot\pi^{p\IZ}\}$ clearly forms a basis of $\pi^{\IZ}/\pi^{p\IZ}$. On the other hand, the elements of $U_F^{(1)}/U_F^{(p+1)}$ can be uniquely represented by elements
		\[
		1+ a_1\pi + \cdots + a_k \pi^k
		\]
		by Corollary~\ref{cor:representation}, where we can choose each $a_i$ as an integer with $0 \leq a_i \leq p -1$ because $F/\IQ_p$ is totally ramified by Lemma~\ref{lem:cycExt}. In particular, this means that $F^\times/F^{\times p}$ has cardinality $p^{p+1}$, which implies the claim about the dimension.
		
		Next, \eqref{eq:F times mod F times p} shows that 
		\begin{align*}
			b \equiv a \pmod{F^{\times p}} \iff ~c \equiv p^{a_0} \pmod{\pi^{p\IZ}} \land~ u \equiv 1+a_1\pi + \cdots + a_p\pi^{p} \pmod{U_F^{(p+1)}}.
		\end{align*}
		By Lemma~\ref{lem:U_Kpowers}, $U_F^{(p+1)}= \big(U_F^{(1)}\big)^p$
		so the second congruence is equivalent to
		\[
		u \equiv 1+a_1\pi + \cdots + a_p\pi^{p} \pmod{\pi^{p+1}}
		\]
		by Remark~\ref{rem:mod_rule}.
	\end{proof}
	
	\section{Proof of the local Kronecker--Weber theorem \label{sec:proof}}
	We shall prove the following explicit version of the local Kronecker--Weber theorem.
	
	\begin{theorem}[Local Kronecker--Weber theorem]
		\label{thm:myKW}
		Let $F/\IQ_p$ be an abelian extension of degree $n= p^l\cdot m$ with integers $l$ and $m$ such that $(p,m) = 1$. Then $F \subseteq \IQ_p(\zeta_{p^{l+2}}, \zeta_{p^{n}-1})$.
	\end{theorem}
	
	Since $F/\IQ_p$ is abelian, its Galois group
	$G=\Gal(F/\IQ_p)$ is finite abelian and hence (by the structure
	theorem for finite abelian groups) decomposes as 
	\[
	G = G_p \times G_{p'},
	\]
	where $G_p$ is the Sylow $p$-subgroup of $G$ and $G_{p'}$ is the product
	of the Sylow $l$-subgroups for all primes $l\neq p$. Let $L_p$ and
	$L_{p'}$ denote the fixed fields of $G_p$ and $G_{p'}$, respectively.
	By the Galois correspondence,
	\[
	[L_p:\IQ_p]=|G_{p'}|=m \qquad\text{and}\qquad
	[L_{p'}:\IQ_p]=|G_p|=p^l,
	\]
	since $L_p/\IQ_p$ has degree prime to $p$, while $L_{p'}/\IQ_p$ has $p$-power degree. Moreover, because $G$ is generated by $G_p$ and $G_{p'}$, we have
	\[
	L_p L_{p'} = F.
	\]
	Consequently, it suffices to prove the theorem in the prime-to-$p$
	and $p$-power cases separately: if both $L_p$ and $L_{p'}$ are contained
	in the cyclotomic extension appearing in Theorem~\ref{thm:myKW}, then so
	is $F$. Following standard terminology, we shall refer to the prime-to-$p$ case as the \emph{tame} case and to the $p$-power case as the \emph{wild} case.

	\subsection{The tame case}
	We first observe that in the notation of Theorem~\ref{thm:myKW}, $\IQ_p(\zeta_{p^m-1}) \subseteq \IQ_p(\zeta_{p^n-1})$ by Lemma~\ref{lem:cycExt} and the uniqueness of inertia fields. In particular,
	\[
	\IQ_p(\zeta_p,\zeta_{p^m-1}) \subseteq \IQ_p(\zeta_{p^{l+2}},\zeta_{p^n-1}).
	\]
	In order to establish Theorem~\ref{thm:myKW} in the tame case, it therefore suffices to show $L_p \subseteq \IQ_p(\zeta_p,\zeta_{p^m-1})$, which is the statement of the following lemma.
	
	\begin{lemma}
		\label{lem:tame KW}
		Let $F/\IQ_p$ be an abelian extension of degree $n$ such that $(n,p)=1$. Then $F \subseteq \IQ_p(\zeta_p, \zeta_{p^{n}-1})$.
	\end{lemma}
	\begin{proof}
		We briefly outline the strategy of the proof. The unramified part of $F/\mathbb{Q}_p$ is cyclotomic by Lemma~\ref{lem:cycExt} and the uniqueness of inertia fields, so the analysis will focus on the ramified part. We show that the ramification is entirely controlled by adjoining $\sqrt[e]{-p}$. In particular, $\mathbb{Q}_p(\sqrt[e]{-p})/\mathbb{Q}_p$ is an intermediate abelian extension, which forces $\zeta_e \in \mathbb{Q}_p$. By Lemma~\ref{lem:cycExt}, this implies $e \mid (p-1)$, giving an upper bound on the ramification index. Finally, we show that $\IQ_p(\sqrt[p-1]{-p}) = \mathbb{Q}_p(\zeta_p)$, which implies that the ramified part is cyclotomic as well.
		
		Let $F_0$ denote the inertia subfield of $F/\IQ_p$. By the uniqueness of inertia fields and Lemma~\ref{lem:cycExt}, $F_0 = \IQ_p(\zeta_{p^{f}-1})$. Let us also consider $E \coloneq \IQ_p(\zeta_{p^{n}-1})$, which is an unramified extension of $\IQ_p$ satisfying $[E:\IQ_p] =n$ and $[E:F_0] =e$ by the tower law and Lemma~\ref{lem:cycExt}.
		
		Let $\pi$ be a uniformiser of $F$. As $F$ is ramified of degree $e$, we have $v_F\left(\frac{\pi^{e}}{-p}\right) = 0$ so $\frac{\pi^{e}}{-p} \in \Oo_F^\times$. Then, since $\overline{F_0} = \overline{F}$, we may find $u  \in \Oo_{F_0}^\times$ such that $\overline{u} = \overline{\frac{~\pi^{e}}{-p}~}$. In particular, $\frac{\pi^{e}}{-pu} \in U_F^{(1)}$. Now, applying Hensel's Lemma to the polynomials $X^{e}-u'$ for any $u' \in U_F^{(1)}$ (note that $(e,p)=1$), we see that $U_F^{(1)} \leq F^{\times e}$. Hence, $\sqrt[e]{\frac{\pi^{e}}{-pu}} \in F$ and thus $\sqrt[e]{-pu} \in F$. In particular, 
		\[
		v_F\big(\sqrt[e]{-pu}\big) = \frac{1}{e}(v_F(p)+v_F(u))= \frac{1}{e},
		\]
		so we get $F = F_0(\sqrt[e]{-pu})$. 
		
		We now shift our focus to $E$. By the properties of finite fields, every element in $\overline{F_0}$ is an $e$-th power\footnote{Indeed, let $E/\IF_q$ be an extension of finite fields of degree $e$. Then $\IF_q^\times \cong C_{q-1}$, $E^{\times} \cong C_{q^e -1}$, and taking $e$-th powers in $E^{\times}$ induces a subgroup isomorphic to $C_{(q^e-1)/(q^e-1, e)}$. We need to show that the latter group contains $C_{q-1}$, which is equivalent to
			\[
			(q-1)\big | \frac{(q^e-1)}{(q^e-1, e)} \iff (q^e-1, e) \big | \frac{q^e-1}{q-1}=\sum_{i=0}^{e-1} q^i.
			\]
			If we let $d=(q-1,e)$, one sees that it suffices to show that $d$ divides the right-hand side, which is clearly the case since $q^i \equiv 1 \pmod{d}$ for all $0\leq i \leq e-1$.} in $\overline{E}$, so $u$ is an $e$-th power in $E$ by Hensel's Lemma. In particular, this implies that 
		\[
		EF = E(\sqrt[e]{-pu}) = E(\sqrt[e]{-p}).
		\]
		$EF/\IQ_p$ is abelian as the compositum of two abelian extensions, which means that the intermediate extension $\IQ_p(\sqrt[e]{-p})$ must be Galois. For this to be the case, we need $\zeta_{e} \in \IQ_p$. However, by Corollary~\ref{cor:rootsOfUnity}, this implies that $e \mid (p-1)$, and hence, $E(\sqrt[e]{-p}) \subseteq E(\sqrt[p-1]{-p})$. The following diagram gives an overview of the different extensions discussed so far:
		\begin{figure}[H]
			\centering
			\begin{tikzcd}
				&E(\sqrt[p-1]{-p})\arrow[d, dash]&\\
				&EF&\\
				E\arrow[ur, dash]&&\arrow[ul, dash]F \arrow[d, dash, "e"]\\
				&&F_0 \arrow[d, dash, "f"] \arrow[dash, llu, "e"']\\
				&&\IQ_p \arrow[uull, dash, "n"]
			\end{tikzcd}
		\end{figure}
		We claim $\IQ_p(\sqrt[p-1]{-p}) = \IQ_p(\zeta_p)$ from which the result will follow since then
		\[
		F \subseteq EF = E(\sqrt[e]{-p}) \subseteq E(\sqrt[p-1]{-p}) = \IQ_p(\zeta_{p^{n}-1})(\sqrt[p-1]{-p})= \IQ_p(\zeta_{p^{n}-1}, \zeta_p).
		\]
		To show that $\IQ_p(\sqrt[p-1]{-p}) = \IQ_p(\zeta_p)$, we first observe that if $w$ denotes the unique extension of $v_p$ to $\IQ_p(\sqrt[p-1]{-p})$, then $w\big(\sqrt[p-1]{-p}\big) = \frac{1}{p-1}$. It follows that
		\[
		[\IQ_p(\sqrt[p-1]{-p}) : \IQ_p] = p-1 = [\IQ_p(\zeta_p) : \IQ_p],
		\]
		where the second equality follows from Lemma~\ref{lem:cycExt}. Consequently, it suffices to show that $\IQ_p(\sqrt[p-1]{-p}) \subseteq \IQ_p(\zeta_p)$, or $\sqrt[p-1]{-p} \in \IQ_p(\zeta_p)$. We note that the polynomial
		\[
		X^{p-1}-\frac{-p}{(\zeta_p-1)^{p-1}}
		\]
		reduces to $X^{p-1}-1$ in the residue field $\IF_p$ of $\IQ_p(\zeta_p)$ by Lemma~\ref{lem:p/(1-zeta)^p}. Therefore,
		\[
		\sqrt[p-1]{\frac{-p}{(\zeta_p-1)^{p-1}}} \in \IQ_p(\zeta_p)
		\]
		by Hensel's Lemma, which implies $\sqrt[p-1]{-p} \in \IQ_p(\zeta_p)$ as required.
	\end{proof}
	
	\subsection{The wild case}
	The wild case is more involved. Our goal is to show that the (infinite) maximal abelian pro-$p$ extension of $\IQ_p$ is a cyclotomic extension; the precise calculation will then allow us to read off a relevant cyclotomic extension containing the given finite abelian extension of $\IQ_p$. In a first step, we determine an ``upper bound'' for the Galois group of this maximal abelian pro-$p$ extension of $\IQ_p$. As a base case, we compute the compositum of all $C_p$-extensions of $\IQ_p(\zeta_p)$ that are abelian over $\IQ_p$, and then extend the argument to the infinite case\footnote{We work over $\IQ_p(\zeta_p)$ to apply Kummer theory; this merely potentially enlarges the cyclotomic extensions.}. We proceed to show that the calculated Galois group actually appears as the Galois group of a cyclotomic extension. The concrete calculations will then allow us to explicitly describe for each finite abelian extension a cyclotomic extension containing it.
	
	For ease of notation, we set $K := \IQ_p(\zeta_p)$ and denote by $v$ the extension of $v_p$ to $K$. Observe that $K$ is just the trivial extension for $p=2$. For $p>2$, on the other hand, $K/\IQ_p$ is totally ramified of degree $p-1$ with uniformiser $\pi := \zeta_p-1$ and $\Gal(K/\IQ_p) \cong (\IZ/p\IZ)^\times$ by Lemma~\ref{lem:cycExt}. The latter group is cyclic as the multiplicative group of the field $\mathbb{F}_p$. We let $k$ be a generator of $(\IZ/p\IZ)^\times$, so $\Gal(K/\IQ_p) = \langle \sigma \rangle$, where $\sigma(\zeta_p) = \zeta_p^k$. Crucially, note that $k$ is a primitive root modulo $p$.
	
	Following the structure described above, the first step is the base case, i.e., classifying (the compositum of) the $C_p$-extensions of $K$ that are abelian over $\IQ_p$. The crucial ingredient is Kummer theory, which can be summarised as the following fundamental result.
	
	\begin{theorem}[Kummer theory]
		\label{thm:Kummer theory2}
		Let $F$ be a field with $\charK F \ne p$ and $\zeta_p \in F$. Then there exists a bijection
		\begin{align*}
			\{\text{one-dim.\,subspaces } \langle c \rangle F^{\times p} \le F^\times/F^{\times p}\} & \longleftrightarrow \{\text{$L/F$ Galois extensions of deg.\;$p$}\} \\
			\langle c \rangle F^{\times p} & \longmapsto F(\sqrt[p]{c})\\
			(F^\times \cap L^{\times p})/F^{\times p} & \longmapsfrom L.
		\end{align*}
	\end{theorem}
	
	Kummer theory thus tells us that the $C_p$-extensions of $K$ are in a one-to-one correspondence with the 1-dimensional $\IF_p$-subspaces of $K^\times/K^{\times p}$. When $p>2$, however, not all such extensions are necessarily abelian over $\IQ_p$. Determining which of these $C_p$-extensions remain abelian over $\IQ_p$ requires an analysis of the action of $\sigma$ on elements of the higher unit groups of $K$, as summarised in the next lemma.
	
	\begin{lemma}
		\label{lem:observations}
		Let $a \in \Oo_K$ be such that
		\[
		a \equiv 1 + a_j \pi^j\pmod{\pi^{j+1}}
		\]
		with $a_j \in \Oo_K^{\times}$ (so $a \in U_K^{(j)}$ but $a \notin U_K^{(j+1)}$). Then 
		\[
		\sigma(a) \equiv 1+ k^ja_j\pi^j \pmod{\pi^{j+1}}.
		\]
	\end{lemma}
	\begin{proof}
		First observe that 
		\[
		\frac{\sigma(\pi)}{\pi} = \frac{\sigma(\zeta_p -1)}{\zeta_p-1} = \frac{\zeta_p^k - 1}{\zeta_p-1} = \sum_{i=0}^{k-1}\zeta_p^i = \sum_{i=0}^{k-1}(\pi +1)^i \equiv k \pmod{ \pi},
		\]
		which implies
		\[
		\frac{\sigma(\pi^j)}{\pi^j} = \left(\frac{\sigma(\pi)}{\pi}\right)^j \equiv k^j \pmod{\pi},
		\]
		so
		\begin{equation}
			\label{eq:sigma pi j}
			\sigma\big(\pi^j\big) \equiv \pi^jk^j \pmod{\pi^{j+1}}.
		\end{equation}
		Similarly,
		\[
		\sigma(\pi^{j+1}) \equiv \pi^{j+1}k^{j+1} \pmod{\pi^{j+2}},
		\]
		which implies
		\begin{equation}
			\label{eq:sigma pi j+1}
			\sigma(\pi^{j+1}) \equiv 0 \pmod{ \pi^{j+1}}.
		\end{equation}
		The claim then follows by combining equations~\eqref{eq:sigma pi j} and \eqref{eq:sigma pi j+1}.
	\end{proof}
	
	We can now prove the following characterisation of the $C_p$-extensions of $K$ which are abelian over $\IQ_p$.
	\begin{lemma}
		\label{lem:base_case}
		The compositum $F$ of all $C_p$-extensions of $K$ which are abelian over $\IQ_p$ is
		\begin{itemize}
			\item[(i)] $\IQ_2(\sqrt{-1}, \sqrt{2}, \sqrt{5})$ with $\Gal\big(\IQ_2(\sqrt{-1}, \sqrt{2}, \sqrt{5})/\IQ_2\big) \cong C_2 \times C_2 \times C_2$ for $p=2$;
			\item[(ii)] $\IQ_p(\zeta_{p^2}, \zeta_{p^p -1})$ with $\Gal(\IQ_p(\zeta_{p^2}, \zeta_{p^p -1})/K) \cong C_p \times C_p$ for $p>2$.
		\end{itemize}
	\end{lemma}
	\begin{proof}    
		\emph{(i).} By Kummer theory (Theorem \ref{thm:Kummer theory2}), it is enough to show that $\{-1,2,5\}$ represents a basis of $\IQ_2^\times/\IQ_2^{\times 2}$ as an $\IF_2$-vector space. As $\pi =2$ is a uniformiser for $\IQ_2$, Lemma~\ref{lem:dimOfPowerQuotientWEAK} implies that it suffices to show that $\{-1,5\}$ generates all elements of the shape 
		\[
		y = 1 + 2a_1 + 4a_2
		\] 
		for $a_1, a_2 \in \{0,1\}$ modulo $\IQ_2^{\times 2}$. So let us consider the four different cases for $(a_1, a_2)$.
		\begin{itemize}
			\item $(a_1, a_2) = (0,0)$. Then $y = 1 = (-1)^0 \cdot 5^0$.
			\item $(a_1, a_2) = (0,1)$. Then $y = 5 = (-1)^0 \cdot 5^1$.
			\item $(a_1, a_2) = (1,0)$. Then $y = 3 \equiv -5 = (-1)^1 \cdot 5^1 \pmod{\IQ_2^{\times 2}}$.
			\item $(a_1, a_2) = (1,1)$. Then $y = 7 \equiv -1 = (-1)^1\cdot 5^0 \pmod{\IQ_2^{\times 2}}$. 
		\end{itemize}
		In the last two cases, we used Lemma~\ref{lem:dimOfPowerQuotientWEAK} to reduce modulo $\pi^3 =8$.
		
		\emph{(ii)} We first observe that $\IQ_p(\zeta_{p^2})$ and $\IQ_p(\zeta_{p^p -1})$ are $C_p$-extensions of $K$ that are abelian over $\IQ_p$ by Lemma~\ref{lem:cycExt}, which implies that $F$ contains $\IQ_p(\zeta_{p^2}, \zeta_{p^p -1})$. Lemma~\ref{lem:cycExt} further tells us that $K(\zeta_{p^2})/K$ is totally ramified while $K(\zeta_{p^p-1})/K$ is unramified. This means that both subextensions are linearly disjoint over $K$, which implies $\Gal(K(\zeta_{p^2},\zeta_{p^p-1})/K) \cong C_p \times C_p$. 
		
		It remains to see that any $C_p$-extensions $L/K$ with $L/\IQ_p$ abelian is contained in $\IQ_p(\zeta_{p^2}, \zeta_{p^p -1})$. By Kummer theory, $L = K\big(\sqrt[p]{u}\big)$ for some $u \in K^\times$, where we may represent $u$ as
		\[
		u = p^{a_0}\left(1 + a_1\pi + \cdots +a_p\pi^p\right)
		\]
		with $a_i \in \{0, \dots, p-1\}$ by Lemma~\ref{lem:dimOfPowerQuotientWEAK}.
		
		\textit{Claim. $a_0 = 0$.}
		
		\begin{claimproof}{}
			Let us assume $a_0 \neq 0$. On the one hand,
			\begin{equation*}
				\sigma(u) = p^{a_0}\sigma\left(1 + a_1\pi + \cdots +a_p\pi^p\right).
			\end{equation*}
			On the other hand, as $K\big(\sqrt[p]{u}\big) = K\big(\sqrt[p]{\sigma(u)}\big)$, we have $\sigma(u) \in \langle u \rangle K^{\times p}$ by Kummer theory, say $\sigma(u) \in u^r K^{\times p}$ with $r \in \{1,...,p-1\}$. Hence,
			\[
			\sigma(u) \equiv u^r  = p^{r a_0}\left(1 + a_1\pi + \cdots +a_p\pi^p\right)^r \pmod{K^{\times p}}.
			\]
			Comparing the last two equations and using Lemma~\ref{lem:dimOfPowerQuotientWEAK}, we see $r a_0 \equiv a_0 \pmod{p}$, which implies $r=1$ since $a_0 \neq 0$. In particular, Lemma~\ref{lem:dimOfPowerQuotientWEAK} then further implies
			\[
			\sigma(1 + a_1\pi + \cdots +a_p\pi^p) \equiv 1 + a_1\pi + \cdots +a_p\pi^p \pmod{\pi^{p+1}}.
			\]
			Now let $j \in \{1, \dots, p\}$ be minimal such that $a_j \neq 0$. Combining the previous equation with Lemma~\ref{lem:observations} then implies that 
			\[
			a_j k^j\pi^j \equiv a_j\pi^j \pmod{ \pi^{j+1}},
			\] 
			which is equivalent to $\overline{a_j} = \overline{a_jk^j}$. However, as $k$ is a primitive root modulo $p$, this is impossible for $j < p-1$, so we may write
			\[
			u = p^{a_0}(1+ a_{p-1} \pi^{p-1}+a_p\pi^p).
			\]
			Next, we observe that we can factorise
			\[
			u  = p^{a_0}(1+ a_{p-1} \pi^{p-1}+a_p\pi^p) \equiv p^{a_0}(1+lp)(1+a_p' \pi^p) \pmod{K^{\times p}}
			\]
			for suitable $l,a_p' \in \{0, \dots, p-1\}$. Indeed, using Lemma~\ref{lem:p/(1-zeta)^p} to write $p = - \pi^{p-1} + x\pi^p$ for some $x\in \Oo_K$, we get
			\[
			(1+lp)(1+a_p' \pi^p) \equiv 1 -l\pi^{p-1} + (lx + a_p')\pi^p \pmod{\pi^{p+1}},
			\]
			so choosing $l,a_p' \in \{0, \dots, p-1\}$ such that $\overline{l} = \overline{-a_{p-1}}$ and $\overline{a_p'} = \overline{a_p-lx}$, the claim follows from Lemma~\ref{lem:dimOfPowerQuotientWEAK}.
			
			Let us assume we could show that $K\big(\sqrt[p]{1 + a_p' \pi^p}\big) \subseteq K(\zeta_{p^p-1})$. Then the last equation shows that the extension $K\big(\zeta_{p^p-1}, \sqrt[p]{p^{a_0}(1 +lp)}\big)/\IQ_p$ is the compositum of the two abelian extensions $\IQ_p(\zeta_{p^p-1})/\IQ_p$ and $L/\IQ_p$, and is therefore itself abelian. In particular, the intermediate extension $\IQ_p\big(\sqrt[p]{p^{a_0}(1 +lp)}\big)/\IQ_p$ is Galois. However, as $\IQ_p$ does not contain the $p$-th roots of unity by Corollary~\ref{cor:rootsOfUnity}, this is impossible when $a_0\neq 0$, contradicting our assumption.
			
			The last paragraph shows that in order to complete the proof of $a_0 =0$, it suffices to show that $K\big(\sqrt[p]{1 + a_p' \pi^p}\big) \subseteq K(\zeta_{p^p-1})$. This is clearly the case for $a_p' = 0$, so let us assume $a_p' \neq 0$ and denote $u' = 1 + a_p'\pi^p \in U_K^{(p)}$. Then $u' \notin K^{\times p}$ by Lemma~\ref{lem:dimOfPowerQuotientWEAK}, so $K\big(\sqrt[p]{u'}\big)/K$ is a $C_p$-extension by Kummer theory. 
			In particular, $u'$ is a $p$-th power in $K\big(\sqrt[p]{u'}\big)$, so Lemma~\ref{lem:powersConverse} tells us that we may find $\overline{x}\in \overline{K\big(\sqrt[p]{u'}\big)}$ such that $\overline{x}^p - \overline{x} - \overline{a_p'} = 0$. However, this would be impossible for $\overline{a_p'} \neq 0$ if we had $\overline{K} = \IF_p$, so we conclude that the extension $K(\sqrt[p]{u'})/K$ must give rise to a (degree-$p$) extension of residue fields. In particular, $K\big(\sqrt[p]{u'}\big)/K$ is unramified, so $K\big(\sqrt[p]{u'}\big)= K(\zeta_{p^p-1})$ by Lemma~\ref{lem:cycExt}, as required. This concludes the proof that $a_0 = 0$.
		\end{claimproof}
		
		To complete the proof of Lemma~\ref{lem:base_case}, we calculate all possible values for the remaining $a_i$ and show that in each case, $L\subseteq F$. Let $j \in \{1, \dots, p\}$ be minimal such that $a_j \neq 0$. 
		
		We first consider the case $j =p$, in which case $u=1+ a_p \pi^p$, so the argument in the last paragraph shows that
		\[
		L = \IQ_p(\zeta_p, \sqrt[p]{u}) = \IQ_p(\zeta_p, \zeta_{p^p-1}) \subseteq F
		\]
		for any $a_p \in \{1,\dots, p-1\}$. In particular, this shows that the extension $\IQ_p(\zeta_p, \zeta_{p^p-1})/K$ corresponds to the set $\{(1+ a_p \pi^p)K^{\times p}:a_p\in\{0,\dots, p-1\}\}$ under the Kummer correspondence.
		
		If $j < p$, the previous paragraph and Lemma~\ref{lem:dimOfPowerQuotientWEAK} show that $u\notin \langle u' \rangle K^{\times p}$ for any $u' = 1+ a_p' \pi^p$ with $a_p' \in \{1,\dots, p-1\}$. Using again that $\IQ_p(\zeta_p, \sqrt[p]{u'}) = \IQ_p(\zeta_p, \zeta_{p^p-1})$, we observe that $K\big(\sqrt[p]{uu'}\big)/\IQ_p$ is a subextension of the compositum of the two abelian extensions $L/\IQ_p$ and $\IQ_p(\zeta_{p^p-1})/\IQ_p$, and hence is itself abelian. In particular, $K\big(\sqrt[p]{uu'}\big) = K\big(\sqrt[p]{\sigma(uu')}\big)$ and hence $\sigma(uu') \in \langle uu'\rangle K^{\times p}$ by Kummer theory, so there exists an $l \in \{0,\dots, p-1\}$ such that $\sigma(uu') \in (uu')^lK^{\times p}$.
		
		Using the identity $\sigma(uu') = \sigma(u) \sigma(u')$, we now find a different expression for the domain of $\sigma(uu')$. As we have already observed in the argument for $a_0 =0$, there is an $r\in\{0,\dots,p-1\}$ such that $\sigma(u) \in u^rK^{\times p}$. For the domain of $\sigma(u')$, we use Lemma~\ref{lem:observations} to observe that
		\[
		\sigma(u') = \sigma(1 + a_p'\pi^p) \equiv 1 + k^p a_p' \pi^p \equiv 1+ ka_p'\pi^p \equiv (1+a_p'\pi^p)^k = (u')^k\pmod{\pi^{p+1}},
		\]
		so 
		\[
		\sigma(u') \equiv  (u')^k \pmod{K^{\times p}}
		\]
		by Lemma~\ref{lem:dimOfPowerQuotientWEAK}.
		Hence,
		\[
		\sigma(uu') \in  (uu')^lK^{\times p} \cap u^r\left(u'\right)^kK^{\times p} ,
		\]
		which forces $l = r = k$, so $\sigma(u) \in u^k K^{\times p}$. By Lemma~\ref{lem:dimOfPowerQuotientWEAK} and expanding out, we see that 
		\[
		\sigma(u) \equiv u^k =(1 + a_j\pi^j + \cdots + a_p\pi^p )^k \equiv 1+ka_j\pi^j \pmod{\pi^{j+1}}.
		\]
		On the other hand,
		\[
		\sigma(u) \equiv k^ja_j\pi^j \pmod{\pi^{j+1}}.
		\]
		by Lemma~\ref{lem:observations}. Since $j<p$ and $k$ is a primitive root modulo $p$, this implies $j = 1$, so $a_1 \neq 0$. 
		
		We claim that the value of $a_1$ determines the value of the other $a_i$ for $i< p$. Indeed, on the one hand, expanding out $u^k$ once more, we see that
		\[
		\sigma(u)\equiv (1 + a_1\pi + \cdots + a_p\pi^p )^k \equiv 1+ ka_1\pi+\sum_{i=2}^p (ka_i+f_i(a_1, \dots, a_{i-1}))\pi^i \pmod{\pi^{p+1}},
		\]
		where $f_i \in \IZ[X_1, \dots, X_{i-1}]$ for each $i$. On the other hand, writing $\sigma(\pi) = k\pi + b\pi^2$ for some $b \in \Oo_K^{\times}$, we get
		\[
		\sigma(u) = 1 + \sum_{i = 1}^{p} a_i ( k\pi + b\pi^2)^i \equiv 1 + ka_1\pi +\sum_{i=2}^p (a_ik^i + g_i(a_1, \dots, a_{i-1};b))\pi^i \pmod{\pi^{p+1}},
		\]
		where $g_i \in \IZ[X_1, \dots, X_{i-1}; Y]$. Recursively reducing modulo $\pi^{i+1}$ for $i \in \{2, \dots, p-1\}$, this implies
		\[
		\overline{a_ik^i + g_i(a_1, \dots, a_{i-1};b)} = \overline{ka_i+ f_i(a_1,\dots, a_{i-1})},
		\]
		giving rise to unique values $a_i \in \{0,\dots, p-1\}$ for each $i$ as $k$ is a primitive root modulo $p$. 
		
		For $a_p$, there are $p$ possible different values, which correspond to multiplying $u$ with $u' = 1 +a_p'\pi^p$ for different values of $a_p'$. In particular, modulo these $u'$, there are at most $p$ distinct values of $u$, corresponding to the $p$ different values of $a_1$. As $\sigma(\zeta_p^i) = (\zeta_p^i)^k$ and $\zeta_p = \pi+1$, we see that these $p$ values correspond precisely to the values $\zeta_p^i$ modulo $K^{\times p}$ for $i \in \{1, \dots p-1\}$. Thus, we get that 
		\[
		F \subseteq K\big(\sqrt[p]{\zeta_p u'}\big) \subseteq K\big(\sqrt[p]{\zeta_p}, \sqrt[p]{u'}\big)\subseteq \IQ_p(\zeta_{p^2}, \zeta_{p^p-1})
		\]
		as claimed.
	\end{proof}
	
	From this base case, we can now conclude the infinite case. While this is straightforward for $p>2$, the case $p=2$ requires the following slightly technical lemma:
	
	\begin{lemma}
		\label{lem:Q_2(i)notInC_4}
		$\IQ_2(i)$ is not contained in a $C_4$-extension of $\IQ_2$.
	\end{lemma}
	
	This lemma at hand, we may then deduce:
	
	\begin{corollary}
		\label{cor:upperBound}
		Let $\tilde{F}$ be the unique maximal pro-$p$ extension of $K$ that is abelian over $\IQ_p$. Then $\Gal(\tilde{F}/K)$ 
		\begin{itemize}
			\item[(i)] is a quotient of $C_2\times \IZ_2 \times \IZ_2$ for $p=2$, and
			\item[(ii)] is a quotient of $\IZ_p \times \IZ_p$ for $p>2$.
		\end{itemize}
	\end{corollary}
	\begin{proof}
		\emph{(ii).} By Lemma~\ref{lem:base_case}, the compositum of all $C_p$-extensions of $K$ contained in $\tilde{F}$ is exactly $L := K(\zeta_{p^2},\zeta_{p^p-1})$. As these $C_p$-extensions are the minimal subextensions of $\tilde{F}/K$, the Galois group $\Gal(\tilde{F}/L)$ is exactly the intersection of all maximal subgroups of $\Gal(\tilde{F}/K)$ by the infinite Galois correspondence (cf.~\cite{Szamuely09}, Theorem 1.3.11). In other words, $\Gal(\tilde{F}/L)$ is the Frattini subgroup of $\Gal(\tilde{F}/K)$ and 
		\[
		\Gal(L/K) \cong \Gal(\tilde{F}/K)/\Gal(\tilde{F}/L)
		\] 
		is its Frattini quotient. As $\Gal(L/K) \cong C_p \times C_p$, we see that, as pro-$p$ groups, 
		\[
		\rank(\Gal(\tilde{F}/K)) = \rank(\Gal(L/K)) = \rank(C_p \times C_p) = 2,
		\]
		so $\Gal(\tilde{F}/K)$ must be a quotient of the free abelian pro-$p$ group of rank $2$, which is $\IZ_p \times \IZ_p$.
		
		\emph{(i).} By the same argument as in (ii), we could now conclude that the Galois group of the maximal abelian pro-$2$ extension of $\IQ_2$ is a quotient of $\IZ_2 \times \IZ_2 \times \IZ_2$. However, by Lemma~\ref{lem:Q_2(i)notInC_4}, $\IQ_2(i)$ is not contained in a $C_4$-extension of $\IQ_2$, so $\Gal(\tilde{F}/K)$ must actually be a quotient of $C_2\times \IZ_2 \times \IZ_2$.
	\end{proof}
	
	To show Lemma~\ref{lem:Q_2(i)notInC_4}, we must examine the quadratic extensions of $\IQ_2(i)$. By Kummer theory, these extensions correspond exactly to the 1-dimensional $\IF_2$-subspaces of $\IQ_2(i)^\times/\IQ_2(i)^{\times 2}$. The next lemma calculates a basis of $\IQ_2(i)^\times/\IQ_2(i)^{\times 2}$ as an $\IF_2$-vector space explicitly.
	
	\begin{lemma}
		\label{lem:squaresQ_2(i)}
		$(1+i,1+2i,2,5)$ is a basis of $\IQ_2(i)^\times/\IQ_2(i)^{\times 2}$ as an $\IF_2$-vector space.
	\end{lemma}
	\begin{proof}		
		Let us write $F=\IQ_2(i)$ and denote $v=v_F$. We note that $F/\IQ_2$ is ramified of degree $2$ with uniformiser $\pi = i-1$ since
		\[
		v(\pi) = \frac{1}{2}v(\pi^2)= \frac{1}{2}v(-2i) = \frac{v(2)}{2} + \frac{v(i^2)}{4} = \frac{1}{2}.
		\]
		As a first step, we show that $F^\times/F^{\times 2}$ has $\IF_2$-dimension $4$, for which we generalise the argument in Lemma~\ref{lem:dimOfPowerQuotientWEAK}. We recall that
		\[
		F^\times/F^{\times 2} \cong \pi^{\IZ}/{\pi}^{2\IZ} \times U_F^{(1)}/\big(U_F^{(1)}\big)^2
		\]
		by Fact~\ref{fact:multgrp_direct_prod}, so it suffices to show that
		$U_F^{(1)}/\big(U_F^{(1)}\big)^2$ has $\IF_2$-dimension 3, or equivalently, cardinality 8. Writing 
		\[
		U_F^{(1)}/\big(U_F^{(1)}\big)^2 \cong \frac{ U_F^{(1)}/U_F^{(5)}}{\big(U_F^{(1)}\big)^2/U_F^{(5)}}
		\]
		by Lemma~\ref{lem:U_Kpowers_general} and noting that
		\[
		\left| U_F^{(1)}/U_F^{(5)}\right| = 16
		\]
		by Corollary~\ref{cor:representation}, we observe that it suffices to show
		\[
		\left| \big(U_F^{(1)}\big)^2/U_F^{(5)}\right| = 2.
		\]
		Letting $\Phi: U_F^{(1)}/U_F^{(5)}\rightarrow \big(U_F^{(1)}\big)^2/U_F^{(5)}$ denote the map $x\mapsto x^2$, this is equivalent to showing $|\ker \Phi| = 16/2=8$. 
		
		We note immediately that $\ker \Phi < 16$: For example $\Phi\big((1+\pi)U_F^{(5)}\big)$ is non-trivial as otherwise
		\[
		(1+\pi)^2 \equiv 1 \pmod{\pi^5}
		\]
		by Remark~\ref{rem:mod_rule}, meaning $v( (1+\pi^2)^2 -1) \geq 5v(\pi)$, which contradicts
		\[
		v\left((1+\pi)^2-1\right) =v(\pi^2 +2\pi) = \min\{v(\pi^2), v(2\pi)\}= v(\pi^2) = 2v(\pi).
		\]
		It remains to show $|\ker \Phi| \geq 8$. To do so, we use Corollary~\ref{cor:representation} to represent any of $U_F^{(1)}/U_F^{(5)}$ as 
		\[
		(1+a_1\pi +a_2\pi^2 + a_3\pi^3+a_4\pi^4)U_F^{(5)},
		\]
		with $a_i \in \{0,1\}$ for each $i$; we will show that all elements with $a_1 =0$ are contained in $\ker \Phi$. Indeed, writing $b = a_3\pi^3 + a_4\pi^4 \in (\pi^3)$ to slightly simplify the notation and expanding out, we see (using $a_2^2 = a_2$)
		\begin{equation*}
			(1 + a_2\pi^2 +b)^2 = 1+a_2^2\pi^4 + b^2+2a_2\pi^2 +2b +2a_2\pi^2b \equiv 1+a_2\pi^4+2a_2\pi^2 \pmod{\pi^5},
		\end{equation*}
		where
		\[
		a_2\pi^4+2a_2\pi^2 = a_2\pi^2(\pi^2+2) = a_2\pi^2(-2\pi) \equiv 0  \pmod{\pi^5},
		\]
		so the claim follows from Remark~\ref{rem:mod_rule}.
		
		To show that $\{1+i,1+2i,2,5\}$ forms a basis of $F^\times/F^{\times 2}$, it hence suffices to show that the four elements are linearly independent over $F^{\times 2}$. So assume 
		\[
		x = (1+i)^\alpha\cdot (1 + 2i)^\beta \cdot 2^\gamma \cdot 5^\delta =y^2
		\] 
		for some $y\in F^\times$ and $\alpha, \beta, \gamma, \delta \in \{0,1\}$. Then, denoting by $N$ the field-theoretic norm map of the extension $F/\IQ_2$, 
		\[
		2^\alpha\cdot5^\beta\cdot4^\gamma\cdot 25^\delta = N(x) = N(y^2) = N(y)^2 \in \IQ_2^{\times 2},
		\]
		so $\alpha = \beta = 0$ by the proof of Lemma~\ref{lem:base_case}(i). It remains to exclude the three cases $(\gamma, \delta) = (1,0), (0,1), (1,1)$.
		\begin{itemize}
			\item $(\gamma, \delta) = (1,0)$ would mean that $2 \in F^{\times 2}$, so $F = \IQ_2(i, \sqrt{2})$. But then also $\zeta_8 = \frac{1+i}{\sqrt{2}} \in F$, contradicting the fact that $\IQ_2(\zeta_8)$ is totally ramified of degree $4 \neq 2$ over $\IQ_2$ by Lemma~\ref{lem:cycExt}.
			\item $(\gamma, \delta) = (0,1)$ would mean that $F = \IQ_2(i, \sqrt{5})$. As $5 \equiv -3 \pmod{2^3}$, Lemma~\ref{lem:dimOfPowerQuotientWEAK} further implies that $F= \IQ_2(i, \sqrt{-3})$. In particular, $\zeta_3 = \frac{-1+\sqrt{-3}}{2} \in F$, contradicting the fact that $\IQ_2(\zeta_3)$ is the unique unramified extension of $\IQ_2$ of degree 2 by Lemma~\ref{lem:cycExt}.
			\item $(\gamma, \delta) = (1,1)$ would mean that $\IQ_2(i) = \IQ_2(i, \sqrt{2\cdot 5})$. However, this would imply that $\IQ_2(i, \sqrt{5}) = \IQ_2(i, \sqrt{2})$, which contradicts what we have shown in the previous two items. \qedhere
		\end{itemize}    
	\end{proof}
	
	We can now deduce Lemma~\ref{lem:Q_2(i)notInC_4}.
	\begin{proof}[Proof (Lemma~\ref{lem:Q_2(i)notInC_4}).]
		We continue to write $F = \IQ_2(i)$ and let $L/\IQ_2$ be a Galois extension of degree 4 containing $F$. Then $[L: F] = 2$, so $L = F(\sqrt{x})$ by Kummer theory. By the previous lemma, we may assume $x = (1+i)^\alpha\cdot (1 + 2i)^\beta \cdot 2^\gamma \cdot 5^\delta$ for $\alpha, \beta, \gamma, \delta \in \{0,1\}$. Let us assume we could show $\alpha = \beta = 0$. Then $x$ would be one of 2, 5, 10, and either way $\IQ_2(\sqrt{x})$ is an intermediate extension of degree 2, meaning $L/\IQ_2$ has two distinct proper non-trivial intermediate fields, so it cannot be a $C_4$-extension.
		
		It hence remains to prove $\alpha = \beta = 0$. If $\sigma$ denotes the generator of $\Gal(F/\IQ_2)$ (i.e., $\sigma$ is complex conjugation), then, $F(\sqrt{x}) = F\big(\sqrt{\sigma(x)}\big)$, so $\sigma(x) \in \langle x \rangle F^{\times 2}$ by Kummer theory. In particular, $\sigma(x) \equiv x \pmod{F^{\times 2}}$.
		
		We now consider the different cases for $\alpha$ and $\beta$, starting with $\alpha = \beta =1$. Then
		\[
		x \equiv (1+i)\cdot (1 + 2i) \cdot 2^\gamma \cdot 5^\delta \pmod{F^{\times 2}},
		\]
		and
		\begin{align*}
			\sigma(x)& \equiv (1-i)\cdot (1 - 2i) \cdot 2^\gamma \cdot 5^\delta \\
			&\equiv (1-i)\cdot (1 - 2i) \cdot 2^\gamma \cdot 5^\delta (1+i)^2(1+2i)^2 \\ 
			& \equiv (1+i)\cdot (1 + 2i) \cdot 2^{\gamma + 1} \cdot 5^{\delta + 1} \pmod{F^{\times 2}}.
		\end{align*}
		Together the last two equations imply $0 \equiv 2\cdot 5 \pmod{F^{\times 2}}$, contradicting the linear independence proved in Lemma~\ref{lem:squaresQ_2(i)}. Similarly, we get contradictions $2, 5 \equiv 0 \pmod{F^{\times 2}}$ for $(\alpha, \beta) = (1,0), (0,1)$, respectively.
	\end{proof}
	
	For the next step of our argument, we want to show that $\tilde{F}$ from Corollary~\ref{cor:upperBound} is cyclotomic. To do so, we prove that the groups calculated in Corollary~\ref{cor:upperBound} actually appear as the Galois groups of cyclotomic extensions of $K$.
	\begin{lemma}
		\label{lem:maxcyclotomicp>2}
		Let $F=\IQ_p(\zeta_{p^{p^\infty}-1},\zeta_{p^\infty})$. Then 
		\begin{itemize}
			\item[(i)] $\Gal(F/K) = \Gal(F/\IQ_2) \cong C_2 \times \IZ_2 \times \IZ_2$ for $p=2$, and
			\item[(ii)] $\Gal(F/K) \cong \IZ_p \times \IZ_p$ for $p>2$.
		\end{itemize}
	\end{lemma}
	\begin{proof}
		\emph{(ii).} From infinite Galois theory, we have $\Gal(F/K) \cong \varprojlim \Gal(E/K)$, where $E$ runs over the directed partially ordered set $I$ of finite Galois subextensions of $F/K$ (cf.~\cite{Szamuely09}, Proposition 1.3.5). By definition of $F$, these finite subextensions are of the form $E = \IQ_p(\zeta_{p^{p^n}-1}, \zeta_{p^m})$ for some non-negative integers $m,n$. Now $K(\zeta_{p^{p^n}-1})/K$ is unramified of degree $p^n$ with $\Gal(K(\zeta_{p^{p^n}-1})/K) \cong C_{p^n}$ by Lemma~\ref{lem:cycExt}. $K(\zeta_{p^m})/K$, on the other hand, is totally ramified by Lemma~\ref{lem:cycExt}, and
		\[
		\Gal(K(\zeta_{p^m})/\IQ_p) = \Gal(\IQ_p(\zeta_{p^m})/\IQ_p) \cong (\IZ/p^m\IZ)^\times \cong C_{p^{m-1}(p-1)} \cong C_{p^{m-1}} \times C_{p-1}.
		\]
		Hence, as $(p^{m-1}, p-1) =1$, 
		\[
		\Gal(K(\zeta_{p^m})/K)\cong \Gal(\IQ_p(\zeta_{p^m})/\IQ_p)/\Gal(K/\IQ_p) \cong C_{p^{m-1}}.
		\] Finally, $K(\zeta_{p^{p^n}-1})$ and $K(\zeta_{p^m})$ are linearly disjoint as one extension is unramified and the other one totally ramified. Thus, $\Gal(E/K) = C_{p^n} \times C_{p^{m-1}}$, and the claim follows since $\IZ_p = \varprojlim_k C_{p^k}$.
		
		\emph{(i).} As $\IQ_2(\zeta_{2^\infty})/\IQ_2$ is totally ramified and $\IQ_2(\zeta_{2^{2^\infty}-1})/\IQ_2$ is unramified, we see that both extensions are linearly disjoint over $\IQ_2$, so 
		\[
		\Gal(F/\IQ_2) \cong \Gal(\IQ_2(\zeta_{2^\infty})/\IQ_2) \times \Gal(\IQ_2(\zeta_{2^{2^\infty}-1})/\IQ_2).
		\]
		By Lemma~\ref{lem:cycExt},
		\[
		\Gal(\IQ_2(\zeta_{2^{2^\infty}-1})/\IQ_2) \cong \varprojlim_n \Gal(\IQ_2(\zeta_{2^{2^n}-1})/\IQ_2) \cong \varprojlim_n\, C_{2^n} = \IZ_2.
		\]
		It remains to show that $\Gal(\IQ_2(\zeta_{2^\infty})/\IQ_2)\cong C_2 \times \IZ_2$. Again, we have
		\[
		\Gal(\IQ_2(\zeta_{2^\infty})/\IQ_2) \cong \varprojlim_n \Gal(\IQ_2(\zeta_{2^n})/\IQ_2) \cong \varprojlim_n\, C_{2^n}^{\times} \cong \varprojlim_n(C_2 \times C_{2^{n-2}})
		\]
		by Lemma~\ref{lem:cycExt}, so the claim follows if we can show that the projective limit splits as the product of the inverse limits of the two factors $C_2$ and $C_{2^{n-2}}$. To do so, we analyse the projection maps in the projective system. 
		
		Observe that the elements of $\Gal(\IQ_2(\zeta_{2^n})/\IQ_2)$ are precisely the automorphisms $\sigma_{n,k}$, defined by 
		\[
		\sigma_{n,k}(\zeta_{2^n}) \coloneq \zeta_{2^n}^k, 
		\]
		where $k\in(\IZ/2^n\IZ)^\times$. Then the natural projection
		\[
		\Gal(\IQ_2(\zeta_{2^{n}})/\IQ_2)\longrightarrow\Gal(\IQ_2(\zeta_{2^{n-1}})/\IQ_2)
		\] 
		sends $\sigma_{n,k}$ to $\sigma_{n-1,k}$ since
		\[
		\sigma_{n-1,k}(\zeta_{2^{n-1}})= \sigma_{n,k}(\zeta_{2^{n-1}}) = \sigma_{n,k}\big(\zeta_{2^{n}}^2\big) = \zeta_{2^{n}}^{2k}=\zeta_{2^{n-1}}^k,
		\]
		which, under the isomorphism 
		\[
		\Gal(\IQ_2(\zeta_{2^n})/\IQ_2) \cong (\IZ/2^n\IZ)^{\times},
		\]
		corresponds to the reduction map $k\mapsto k\pmod{2^{n-1}}$  Hence, if we use the standard decomposition $(\IZ/2^n\IZ)^\times\cong\{\pm1\}\times\langle 5\rangle\cong C_2\times C_{2^{n-2}}$, the projection maps in the projective system carry the $C_2$–factor to the $C_2$–factor and the $C_{2^{n-2}}$–factor to the $C_{2^{n-3}}$–factor compatibly, yielding the required decomposition.
	\end{proof}
	By Lemma~\ref{lem:cycExt}, $F$ from the previous lemma is a pro-$p$ extension of $K$ that is abelian over $\IQ_p$. In particular, $F \subseteq \tilde{F}$. Hence, together with Corollary~\ref{cor:upperBound} and the previous lemma, this gives rise to maps
	\[
	C_2\times \IZ_2 \times \IZ_2  \twoheadrightarrow \Gal(\tilde{F}/K) \stackrel{|_F}{\twoheadrightarrow} \Gal(F/K) \cong C_2\times \IZ_2 \times \IZ_2
	\]
	for $p=2$, and
	\[
	\IZ_p \times \IZ_p \twoheadrightarrow \Gal(\tilde{F}/K) \stackrel{|_F}{\twoheadrightarrow} \Gal(F/K) \cong \IZ_p \times \IZ_p
	\]
	for $p>2$. In either case, all epimorphisms are necessarily isomorphisms. Therefore, $F = \tilde{F}$ by the infinite Galois correspondence. 
	
	We can now readily conclude the wild case and thus, our proof of Theorem~\ref{thm:myKW}. Indeed, given an extension $L/\IQ_p$ of degree $p^l$, we can see that the maximal abelian $p$-extension of $K$ of exponent $p^l$ that is also abelian over $\IQ_p$ is contained in
	\[\begin{gathered}        
		\IQ_2(\zeta_4, \zeta_{2^{l+2}}, \zeta_{2^{2^l}-1}) = \IQ_2(\zeta_{2^{l+2}}, \zeta_{2^{2^l}-1}) \text{ for $p=2$, and}\\
		\IQ_p(\zeta_p)(\zeta_{p^{l+1}},\zeta_{p^{p^l}-1})=\IQ_p(\zeta_{p^{l+1}}, \zeta_{p^{p^l}-1}) \text{ for $p>2$}.
	\end{gathered}\]
	Either way, $L \subseteq \IQ_p(\zeta_{p^{l+2}}, \zeta_{p^{p^l}-1})$.
	
	\begin{remark}
		\label{rem:Washington}
		As noted in the introduction, Washington’s book \cite[Chapter~14]{Washington97} contains a proof of the local Kronecker--Weber theorem that is structurally close to the one presented here. We therefore end this paper by comparing the two approaches, clarifying both their common strategy and conceptual differences.
		
		At a high level, both proofs proceed by decomposing a finite abelian extension $L/\IQ_p$ into its tamely ramified and wildly ramified parts. In the tame case, both arguments essentially follow the strategy outlined in the first paragraph of the proof of Lemma~\ref{lem:tame KW}. The treatment above is slightly more streamlined than Washington’s, as it develops the necessary constructions within a unified framework. Washington instead first uses Krasner’s lemma to show that $L$ can be written as $L = L_0(\sqrt[q^r]{a})$ for some $a$ \cite[Lemma~14.5]{Washington97}, and only later proves that one may take $a = -up$ \cite[p.~324]{Washington97}. Moreover, his argument does not make explicit the cyclotomic extensions required at each step; for example, in passing from $\sqrt[e]{-pu}$ to $\sqrt[e]{-p}$, it is not specified which roots of unity must be adjoined. As a result, Washington’s proof does not yield an explicit description of the relevant cyclotomic extensions, unlike ours. 
		
		Moreover, another methodological difference that already becomes apparent in the tame case is that Washington frequently employs analytic arguments. For instance, Washington shows that $\IQ_p(\sqrt[p-1]{-p}) = \IQ_p(\zeta_p)$ using analytic arguments that rely on the completeness of $\IQ_p$ (cf.\ \cite[Lemma~14.6]{Washington97}), whereas our proof relies exclusively on algebraic and valuation-theoretic methods.
		
		Even more substantial differences arise in the treatment of the wildly ramified case. Washington distinguishes between the cases $p \neq 2$ and $p = 2$, analysing them separately in \cite[Case~II, p.~325]{Washington97} and \cite[Case~III, p.~328]{Washington97}. When $p \neq 2$, his argument shows that any abelian extension not contained in a field of the form $\IQ_p(\zeta_{p^{m+1}}, \zeta_n)$ gives rise to an extension $N/\IQ_p$ with $\Gal(N/\IQ_p) \cong C_p^3$. To exclude this possibility, Washington employs a strategy closely related to the one used in Lemma~\ref{lem:base_case}: after adjoining $\zeta_p$, Kummer theory over $K=\IQ_p(\zeta_p)$ is applied, and the action of $\Gal(K/\IQ_p)$ on $K^\times/K^{\times p}$ is analysed in order to determine all possible generators explicitly. A key technical input in this argument is the statement proved here in Lemma~\ref{lem:U_Kpowers}, which Washington establishes using the logarithm on local fields.
		
		The case $p = 2$ is treated in a similar spirit. Here, Washington identifies extensions $N/\IQ_2$ with $\Gal(N/\IQ_2) \cong C_2^4$ and $\Gal(N/\IQ_2) \cong C_4^3$ as the two problematic cases that must be excluded. The former is ruled out using his analogue of Lemma~\ref{lem:U_Kpowers}, while the latter is excluded by observing that the equation $X^2 + Y^2 = -1$ has no solution in $\IQ_2$.
		
		In contrast, the approach taken in this paper treats the wild case for all primes $p$ within a single algebraic framework. By computing explicit representatives for the relevant quotients of unit groups and analysing their behaviour under the Galois action, we are able to determine explicitly the compositum of all $C_p$-extensions of $K$ that are abelian over $\IQ_p$ (Lemma~\ref{lem:base_case}), a point that is only implicit in Washington’s treatment for $p\neq 2$ and is not carried out for $p = 2$. The additional subtleties arising when $p = 2$ therefore emerge naturally from these computations, rather than necessitating a separate structural argument. An additional advantage of this method is that it allows us to determine the relevant absolute Galois group explicitly.
		
		In summary, the proof presented here relies exclusively on algebraic and valuation-theoretic arguments, together with explicit calculations in unit groups, and avoids analytic tools such as logarithmic or exponential series on local fields, which leads to a streamlined and explicit proof of the local Kronecker--Weber theorem.
	\end{remark}
	
\section*{Acknowledgments}
We are grateful to Edern Gillot for suggesting a simpler argument in the proof of Lemma~\ref{lem:maxcyclotomicp>2}, and to Arno Fehm for pointing out an inaccuracy in an earlier version of the proof of Lemma~\ref{lem:tame KW}. We also thank Leo Gitin for valuable feedback and many helpful comments that improved the presentation throughout the paper.
\newpage

\printbibliography

\end{document}